\documentclass[reqno, 12pt]{amsart}

\usepackage{amsmath}
\usepackage{color}
\usepackage{amssymb}
\usepackage{enumerate}
\usepackage[hdivide={3cm,*,3cm}]{geometry}

\newcommand{\N}{\mathbb{N}}

\usepackage{color}

\newcommand{\Const}{\operatorname{Const}}

\newtheorem{theorem}{Theorem}
\newtheorem{prop}{Proposition}

\newtheorem{lm}{Lemma}
\newtheorem{remark}[theorem]{Remark}

\newtheorem{df}{Definition}
\newtheorem{cor}{Corollary}

\DeclareMathOperator{\cl}{cl}

\title{Sequence spaces of Baire-one functions and their porosity}

\author{Marek Balcerzak}
\address{Institute of Mathematics,
         Lodz University of Technology, al. Politechniki 8,
         PL-93-590 \L\'od\'z,
          Poland}
         \email{marek.balcerzak@p.lodz.pl}

\author{Gertruda Ivanova}
\address{Faculty of Mathematics and Computer Science,
      University of {\L}\'{o}d\'{z},
      Banacha 22, PL-90-238 {\L}\'{o}d\'{z}, Poland}
\email{gertruda.ivanova@math.uni.lodz.pl}

\date{}



\begin{document}

\begin{abstract}
We consider several families of sequences of Baire-one self-maps of $[0,1]$ investigated by
A. Alikhani-Koopaei. They form a chain of inclusions where each space is closed and nowhere dense in the next one. We prove that each space is lower $2/3$-porous in the next one. If one modifies these families where sequences of Baire-one functions are real-valued uniformly bounded, then they are Banach spaces and each space of the respective chain is c-porous in the next one.
\end{abstract}

\maketitle

\section{Introduction}
The aim of this paper is to extend and strengthen some results of the articles \cite{Ali1} and \cite{Ali}
by A. Alikhani-Koopaei that concern sequence spaces of Baire-one functions. It was proved in \cite{Ali} that several
spaces are nowhere dense in some larger ones. We are going to replace nowhere denseness by a sharper notion of porosity.

Let $(X, \varrho)$ be a compact metric space and let $B_1(X, X)$ denote the collection of Baire class-one self-maps on $X$. 
As in \cite{Ali}, a non-autonomous $B_1$-dynamical system, or $B_1$-NDS, is defined as a sequence of  $B_1(X, X)$. Let $\mathcal{F}$ be the family of all such systems. Then for any $F = (f_n)_{n\in \mathbb N} \in \mathcal{F}$ and $G = (g_n)_{n\in \mathbb N}\in \mathcal{F}$, we define a metric by
$$\varrho_{\mathcal{F}}(F, G) = \sup_{n \in \mathbb{N}} \varrho_u(f_n, g_n),$$
where $\varrho_u$ is the uniform metric on $B_1(X, X)$ (that is, $\varrho_u(f,g):=\sup_{x\in X} \varrho(f(x),g(x))$
for $f,g\colon X\to X$).
Denote by $\mathcal F(X)$ the family of sequences in $\mathcal F$ that are pointwise convergent. We start with a general folklore like property.

\begin{prop} \label{PP1}
The space $(\mathcal F,\varrho_{\mathcal F})$ is complete and its subspace $\mathcal F(X)$ is closed.
\end{prop}
\begin{proof}
Let $(F_\ell)_{\ell\in\mathbb N}$ be a Cauchy sequence in $\mathcal F$.
Hence for each $k\in\mathbb N$ there exists $m_k\in\mathbb N$ such that 
$\varrho_\mathcal F(F_\ell,F_{m_k})<\frac{1}{k}$ for all $\ell>m_k$.  Let $F_\ell=(f_n^\ell)_{n\in\mathbb N}$ for $\ell\in\mathbb N$ where $f^\ell_n$ is a Baire-one self-map on $X$.
Consequently, we have
\begin{equation} \label{Eq}
\varrho(f_n^\ell(x),f_n^{m_k}(x))<\frac{1}{k}
\end{equation}
for all $n\in\mathbb N$, $x\in X$ and $\ell>m_k$. Hence for any $n\in\mathbb N$ and $x\in X$, the sequence $(f_n^\ell(x))_{\ell\in\mathbb N}$ is Cauchy in $X$. Thus the limit $f_n(x):=\lim\limits_{\ell\to\infty}f_n^\ell(x)$ exists. Taking $\ell\to\infty$ in (\ref{Eq}), we obtain
$$\varrho(f_n(x),f_n^{m_k}(x))\le\frac{1}{k}$$
for all $n\in\mathbb N$ and $x\in X$. 
Hence 
$$\varrho(f_n(x),f_n^{\ell}(x))<\frac{2}{k}$$
for all $n\in\mathbb N$ and $x\in X$. Thus for each $n\in\mathbb N$, the function $f_n$ is a uniform limit of the sequence $(f_n^\ell)_{\ell\in\mathbb N}$ of Baire-one self-maps on $X$. So, $f_n$ is also Baire-one.
Put $F:=(f_n)_{n\in\N}$. We have 
$\varrho_{\mathcal F}(F_\ell,F)<\frac{2}{k}$ for all $\ell>m_k$. Hence $F_\ell\to F$ in $\mathcal F$. 

To show the second assertion, let $(F_\ell)_{\ell\in\mathbb N}$, with $F_\ell=(f_n^\ell)_{n\in\mathbb N}$ for $\ell\in\mathbb N$, be a sequence in $\mathcal F(X)$ convergent to $F=(f_n)_{n\in\mathbb N}\in\mathcal F$.
It is enough to show that $(f_n)_{n\in\mathbb N}$ is pointwise convergent. Fix $x\in X$. We will prove that $(f_n(x))_{n\in\mathbb N}$ is Cauchy in $X$. Let $k\in\mathbb N$ be arbitrary. Clearly, 
$(F_\ell)_{\ell\in\mathbb N}$ is Cauchy in $\mathcal F$, so it satisfies condition (\ref{Eq}) for all $n\in\mathbb N$ and $x\in X$. Since $F_{m_k}\in\mathcal F(X)$,
the sequence
$(f_n^{m_k})_{n\in\mathbb N}$ is pointwise convergent. So, we pick $s_k\in\mathbb N$ such that
$\varrho(f_n^{m_k}(x),f_p^{m_k}(x))<\frac{1}{k}$ for all $n,p>s_k$. Now, for all $n,p>s_k$ we have
$$\varrho(f_n(x),f_p(x))\le\varrho(f_n(x),f_n^{m_k}(x))+\varrho(f_n^{m_k}(x),f_p^{m_k}(x))+\varrho(f_p^{m_k}(x),f_p(x))
<\frac{3}{k}$$
which yields the desired Cauchy condition.
\end{proof}

From now on, we will assume that $X=I$ where $I$ is the unit interval $[0,1]$. 

Recall that a sequence $(f_n)_{n\in\mathbb N}$ of functions from $I$ to $\mathbb R$ is equi-$B_1$ if for every $\varepsilon > 0$, there exists a mapping (a gauge)
$\delta\colon I \to \mathbb{R}^+$ such that for all $n\in\mathbb N$, the condition $|x-y| < \min\{\delta(x), \delta(y)\}$ implies $|f_n(x)- f_n(y)| < \varepsilon$ for any $x, y \in I$.

Note that every term $f_n$ of an equi-$B_1$ sequence is a Baire-one function, by the characterization of 
Lee Tang and Zhao \cite{LTZ}.
In fact this characterization inspired Lecomte \cite{Le} to introduce the notion of an equi-$B_1$ family of functions. Independently, the same notion was created later by Alikhani-Koopaei \cite{Ali2}. Properties of equi-$B_1$ families have been investigated recently in \cite{BHH}, \cite{BHKS}, \cite{BKS}.

A. Alikhani-Koopaei in \cite{Ali} considered the following subfamilies of $\mathcal{F}$:

\begin{itemize}
    \item \textbf{$\mathcal{F}_u(I)$} -- the family of all uniformly convergent sequences with Baire-one terms;
    \item \textbf{$\mathcal{F}_{eq}(I)$} -- the family of pointwise convergent equi-$B_1$ sequences;
    \item \textbf{$\mathcal{F}_{p.w.}(I)$} -- the family of all sequences of Baire-one functions that converge pointwise and their limits are Baire-one functions;
    \item \textbf{$\mathcal{F}(I)$} -- the family of all Baire-one sequences that converge pointwise.
\end{itemize}

It is known (see \cite{Ali1}, \cite{Ali}) that the following inclusions hold
\begin{equation} \label{inclu}
\mathcal{F}_u(I)\subset \mathcal{F}_{eq}(I)\subset \mathcal{F}_{p.w.}(I)\subset \mathcal{F}(I).
\end{equation}

In \cite{Ali1}, it is proved that the first three families of sequences are closed subsets of the fourth. In \cite{Ali}, it is also shown that the first family is nowhere dense in the second, and the second is so in the third.

Note that the metric $\varrho_{\mathcal F}$ in the space $\mathcal F$ with $X=I$ is of the form
$$\varrho_{\mathcal F}(F,G)=\sup_{n\in\mathbb N}\sup_{x\in I}\mid f_n(x)-g_n(x)\mid$$
where $F=(f_n)_{n\in\mathbb N}$, $G=(g_n)_{n\in\mathbb N}$. In fact, the space $\mathcal F$ can be embedded in a normed linear space $b\mathcal F$ of all uniformly bounded sequences of Baire-one functions from $I$ to $\mathbb R$
equipped with the norm
$$\| F\|_\infty:=\sup_{n\in\mathbb N}\sup_{x\in I}\mid f_n(x)\mid$$
where $F=(f_n)_{n\in\mathbb N}$. Thus $F\in b\mathcal F$ whenever there exists $M>0$ such that $\mid f_n(x)\mid\, \le M$ for all $n\in\mathbb N$ and $x\in I$.
In consequence, $\varrho_{\mathcal F}(F,G)=\| F-G\|_\infty$ for $F,G\in\mathcal F$. Therefore, through the paper, we will  use the norm $\|\cdot\|_\infty$ instead of the metric $\varrho_\mathcal F$.
The normed space $b\mathcal F$ and its linear subspaces will be studied in Section 3.


\section{Porosity results for subspaces of $\mathcal F$}
Porosity and $\sigma$-porosity of subsets of $\mathbb R$ or a metric space are notions of smallness
stronger than nowhere denseness. It was used and investigated by several authors.  For extensive surveys on this topic, see \cite{Zaj} and \cite{Zaj1}.

Let $X$ be an arbitrary metric space. By $B\left(x,r\right)$ we denote the open ball centred at $x$ with a radius $r>0$. Let $B(x,0):=\emptyset$.
Fix $M\subset X,$ $x\in X$ and $r>0$. We define
\begin{equation*}
\gamma\left( x,r,M\right): =\sup\{t\geq 0:\exists_{z\in X} \;B\left(z,t\right)\subset B\left(x,r\right)\setminus M\}.
\end{equation*}
Subsequently, let

\begin{equation*}
\underline{p}\left(M,x\right):=2\liminf\limits_{r\rightarrow 0^{+}}\frac{\gamma\left(x,r,M\right)}{r}
\end{equation*}
and
\begin{equation*}
\underline p\left(M\right):=\inf\{\underline p\left(M,x\right)\colon x\in M\}.
\end{equation*}

\begin{df} 
We say that $M$ is lower porous whenever $\underline p\left(M,x\right)>0$ for each $x\in M$. Given
$q\in\left(0,1\right]$, the set $M$ is called lower $q$-porous if $\underline p\left(M\right)=q$. The set $M$ is at least lower $q$-porous if $\underline p\left(M\right)\geq q.$ One can consider the case where $M\subset Z\subset X$ and then the lower porosity relative to the subspace $Z$ can be studied. Then it is reasonable to use notation 
$\underline{p}_Z (M,x)$ and $\underline{p}_Z(M)$.
\end{df}

Let us mention that many porosity phenomena in various function classes were studied in \cite{FIW}, \cite{IWK}, \cite{IW}.
We are going to improve the results of \cite{Ali} for subspaces of $\mathcal F(I)$ in the language of porosity. We will start from a general observation.

\begin{lm} \label{L_e}
Let $\Const$ consist of sequences in $\mathcal F$ whose terms are constant equal to the same 
constant $c\in I$.  Let $M, Z$ be subfamilies of $\mathcal F$ such that
\begin{itemize}
\item $\Const\subset M\subset Z$, and
\item $F+H\in Z$ for all $F\in Z$ and $H\in\Const$ with $(F+H)[I]\subset I^\mathbb N$.
\end{itemize}
Then for each $r\in (0,1)$ we have $\gamma({\mathbf 0},r,M)\le r/3$ where $\mathbf 0$ is the zero sequence. Consequently, $\underline{p}_Z(M)\le 2/3$.
\end{lm}
\begin{proof}
Consider a ball $B(\mathbf 0, r)$ in $Z$ with $r\in (0,1)$, and an arbitrary ball 
$B(G,s)\subset B(\mathbf 0,r)$ with $G\in Z$. 

We have $\|G\|_\infty <r$. Let us show that $\|G\|_\infty\le r-s$. Suppose that $\|G\|_\infty>r-s$. Then putting $t:=r-\|G\|_\infty>0$, we have
$t<s$. The terms of the sequence $G+\mathbf t$ have the values in $[t,r]\subset I$ where 
$\mathbf t:=(t,t,\dots)$. Hence $G+\mathbf t\in Z$. Also, $t<s$ implies that
$$G+\mathbf t\in B(G,s)\subset B(\mathbf 0, r), \mbox{ thus }\|G+\mathbf t\|_\infty < r.$$
On the other hand,
$$\|G+\mathbf t\|_\infty=\|G\|_\infty +t=r,$$
by the choice of $t$. A contradiction.

Now, define 
$$\mathbf c:=((r-s)/2,(r-s)/2,\dots).$$
Then $\mathbf c\in M$,
 and $\|G\|_\infty \le r-s$ implies that $\| G-\mathbf c\|_\infty\le (r-s)/2$.  
Assume that $s>r/3$. Then $(r-s)/2<s$ which implies that $\mathbf c\in M\cap B(G,s)$. Hence $B(G,s)$ is not disjoint from $M$. Consequently, $\gamma({\mathbf 0},r,M)\le r/3$. Since $\mathbf 0\in\Const\subset M$,
it follows that $\underline{p}_Z(M)\le\underline{p}_Z(M,\mathbf 0)\le 2/3$ as desired.
\end{proof}

In the following lemma we use \cite[Example 3.2]{Ali}.

\begin{lm} \label{L1}
Let $a_i:=1/i$ for $i\in\N$. Fix $\alpha\in(0,1/2)$ and $s\colon\mathbb N\to\{ -1,1\}$. Define $f^{\alpha, s}_n\colon I\to\mathbb R$ by
\[
f_n^{\alpha,s}:=\alpha\sum_{i=1}^ns(i)\chi_{\{a_i\}}\mbox{ for }n\in\mathbb N .
\]
Then the sequence $F^{\alpha,s}:=(f^{\alpha,s}_n)_{n\in\mathbb N}$ is equi-$B_1$. Also, it is pointwise
convergent to a Baire-one function, but it is not uniformly convergent.
Moreover, there is a gauge $\delta\colon I\to\mathbb R_+$ such that for any distinct $x,y\in I$, if 
$|x - y| < \min\{\delta(x), \delta(y)\}$, then $f_n^{\alpha,s}(x)=f_n^{\alpha,s}(y)=0$ for all $n\in\mathbb N$.
Also, $\delta$ does not depend on $\alpha$ and $s$.
\end{lm}

\begin{proof}
Clearly, the functions $f_n^{\alpha,s}$, $n\in\mathbb N$, are Baire-one.
Define a function $\delta\colon I \to \mathbb{R}^+$ as follows
$$
\delta(x) = 
\begin{cases} 
\frac{1}{2} (a_{i} - a_{i+1}), & \text{if } x = a_i, \quad i\in\mathbb N, \\ 
\frac{1}{2} \min \left\{x - a_{i+1}, a_{i} - x \right\}, & \text{if } x \in (a_{i+1}, a_i), \quad i\in\mathbb N, \\
1, & \text{otherwise.} 
\end{cases}
$$
Assume that $x, y\in I$ and 
\begin{equation} \label{eq1}
|x - y| < \min\{\delta(x), \delta(y)\}.
\end{equation} 
First suppose that $x\in (0,1]$. Consider two cases.

If there exists $i\in\mathbb N$ such that $x = a_i$,  then we analyse various locations of $y\neq x$ and we deduce that $\delta(y) \le (1/2)\mid a_{i} - y\mid \,
= (1/2)\mid x - y\mid<\delta(x)$. So, by (\ref{eq1}), we get $y = a_i$.

If there exists $i\in\mathbb N$ such that $x\in (a_{i+1}, a_i)$, then we obtain $y\in (a_{i+1}, a_i)$. In this case, for each $n\in\mathbb N$ we have $f_n^{\alpha,s}(x) = f_n^{\alpha,s}(y) = 0$, so $\mid f_n^{\alpha,s}(x) - f_n^{\alpha,s}(y)\mid \, = 0 < \varepsilon$ for each $\varepsilon >0$.

The proof when $y\in(0,1]$ is analogous. The case $x=y=0$ is trivial. Summing up, we obtain that 
$F^{\alpha,s}$ is equi-$B_1$. Moreover, the gauge $\delta$ is as desired.

Define $f^{\alpha,s}\colon I \to \mathbb R$ by
\begin{equation} \label{lab}
f^{\alpha,s}: =\alpha\sum_{i=1}^\infty s(i)\chi_{\{a_i\}}.
\end{equation} 
Then $f^{\alpha,s}$ is Baire-one (since the set of discontinuity points is countable) and
$$\lim_{n \to \infty} f_n^{\alpha,s}(x) = f^{\alpha,s}(x) \quad \text{for each } x \in I.$$
On the other hand, for each $n\in\mathbb N$ we have
$$\sup_{x \in I} |f_n^{\alpha,s}(x) - f^{\alpha,s}(x)| = \alpha>0,$$
so $F^{\alpha,s}$ is not uniformly convergent. 
\end{proof}

\begin{remark} \label{rema}
Assume that $F=(f_n)_{n\in\mathbb N}$ and $G=(g_n)_{n\in\mathbb N}$ are two sequences of equi-$B_1$ functions from $I$ to $\mathbb R$. Then $F+G$ is also equi-$B_1$ family. Indeed, given $\varepsilon >0$ we pick the respective gauges $\delta_F,\delta_G\colon I\to\mathbb R_+$ for $\varepsilon/2$ and $F,G$, respectively. Then 
$\delta:=\min\{\delta_F,\delta_G\}$ is good for $F+G$ and $\varepsilon$. Similarly, we find the respective gauge for $\varepsilon$ and an arbitrary $c\in\mathbb R$ to show that $c F$ forms an equi-$B_1$ family. 
\end{remark}

The following result is an improvement of \cite[Theorem 3.3]{Ali}.

\begin{theorem} \label{T1}
The family $\mathcal{F}_u(I)$ is lower $2/3$-porous in $\mathcal{F}_{eq}(I)$.
\end{theorem}

\begin{proof}
Using Lemma \ref{L_e} with $M:=\mathcal F_u(I)$ and $Z:=\mathcal F_{eq}(I)$, we obtain 
$\underline{p}_{{\mathcal F}_{eq}(I)}(\mathcal F_u(I))\le 2/3$.

It suffices to prove that $\mathcal{F}_u(I)$ is at least lower $2/3$-porous in $\mathcal{F}_{eq}(I)$. To accomplish this, fix a sequence $F \in \mathcal{F}_u(I)$, $F = (f_n)_{n \in \mathbb{N}}$. 
Assume that $r \in (0, 1/2)$ and $\varepsilon \in (0, r/3)$ are arbitrary. Let $f \in B_1(I, I)$ be such that
$$\lim_{n \to \infty} f_n(x) = f(x) \quad \text{for each } x \in I.$$
Since $F\in\mathcal F_u(I)$, pick $n_0$ such that for each $n \ge n_0$ we have
\begin{equation} \label{T1_2}
\sup_{x\in I}\mid f_n(x) - f(x)\mid\, < \varepsilon.
\end{equation}
Let us use Lemma \ref{L1}.
Put $\alpha:=2r/3$ and define $s\colon\mathbb N\to\{-1,1\}$ as follows:
$$s(i):=\begin{cases}
1, & \text{if } f(a_i)\le \frac{1}{2}, \\ 
-1, & \text{if } f(a_i) >\frac{1}{2}. \\ 
\end{cases}
$$
Let $(f_n^{\alpha,s})_{n \in \mathbb{N}}$ be as in Lemma \ref{L1}. 

For each $n \in \mathbb{N}$, we define $g_n\colon I\to \mathbb R$ as follows.
Let $g_n:=f_n$ for $n< n_0$, and $g_n:=f_n+f_n^{\alpha,s}$ for $n\ge n_0$. One can check that the values of $g_n$ are in $I$.
Indeed, let $n\ge n_0$ and consider a fixed $i\le n$. Suppose that $f(a_i)\le 1/2$. Then 
$$g_n(a_i)=f_n(a_i)+\alpha<f(a_i)+\varepsilon+1/3\le 1/2+1/6+1/3=1$$
since $\alpha<2r/3<1/3$ and $\varepsilon<r/3<1/6$. The remaining cases are similar.
Observe that $(g_n)_{n\in\mathbb N}$ is pointwise convergent to the Baire-one function $f+f^{\alpha,s}$ where $f^{\alpha,s}$ is given by (\ref{lab}).

Let us prove that $G:=(g_n)_{n\in\mathbb N} \in \mathcal{F}_{eq}(I)$.
For any $\eta>0$ choose a gauge $\delta_0$ witnessing that $(f_n)_{n\in\mathbb N}$ is an equi-$B_1$ sequence, and let $\delta$ be the respective gauge chosen for $(f_n^{\alpha,s})_{n\in\mathbb N}$ in Lemma \ref{L1}. Observe that the gauge $\sigma:=\min\{\delta_0,\delta\}$ is good for $G$. Indeed, let
$x,y\in I$ be distinct such that $\mid x-y\mid \,<\min\{\sigma(x),\sigma(y)\}$. Then
$\mid x-y\mid \,<\min\{\delta(x),\delta(y)\}$ which (by Lemma \ref{L1}) implies that 
$f_n^{\alpha,s}(x)=f_n^{\alpha,s}(y)=0$ for all $n\in\mathbb N$. Since also
$\mid x-y\mid \,<\min\{\delta_0(x),\delta_0(y)\}$, we have
$\mid g_n(x)-g_n(y)\mid\, = \, \mid f_n(x)-f_n(y)\mid\,
 < \eta$ for all $n\in\mathbb N$, as desired.

Note that $\|G-F\|_\infty=\alpha$, hence $B(G,r/3)=B(G,\alpha/2)\subset B(F,r)$.

We will show that $B(G, \alpha/2 - \varepsilon) \cap \mathcal{F}_u(I) = \emptyset$. So, fix $H \in B(G, \alpha/2 - \varepsilon) \cap \mathcal{F}_{eq}(I)$ where $H = (h_n)_{n \in \mathbb{N}}$ is pointwise convergent to $h\colon I\to I$. 
Fix any $k,n\in\mathbb N$, $k>n\ge n_0$. We have
 $$\alpha\,=\,\mid f^{\alpha,s}_k(a_{n+1}) - f^{\alpha,s}_n(a_{n+1})\mid\;
 \le\; \mid g_k(a_{n+1})-g_n(a_{n+1})\mid\,+\,\mid f_k(a_{n+1})-f_n(a_{n+1})\mid$$
 $$\le\; \mid h_k(a_{n+1})-h_n(a_{n+1})\mid\, + \, 2\|H-G\|_\infty\, +\, \mid f_k(a_{n+1})-f_n(a_{n+1})\mid .$$
 Taking $k\to\infty$ we get
 $$\alpha\le\; \mid h(a_{n+1}) -h_n(a_{n+1})\mid\, +\, \mid f(a_{n+1})-f_n(a_{n+1})\mid\, +\, 
 2\|H-G\|_\infty.$$
 Hence by (\ref{T1_2}) we obtain for each $n\ge n_0$,
 $$\sup_{x\in I}\mid h(x)-h_n(x)\mid\;\ge\;\alpha-2\|H-G\|_\infty-\sup_{x\in I}\mid f(x)-f_n(x)\mid\;\ge\alpha-\alpha+2\varepsilon -\varepsilon=\varepsilon.$$
 Thus $H \notin \mathcal{F}_u(I)$. 
 
Observe that our choice of $G$ depends on $\varepsilon$. So, in fact $G= G_\varepsilon$. Consequently, 
 $$B(G_\varepsilon, \alpha/2-\varepsilon) \cap \mathcal{F}_{eq}(I) \subset B(F,r)\cap\mathcal{F}_{eq}(I) \setminus \mathcal{F}_u(I)$$ 
 which implies that for each $F\in\mathcal F_u(I)$ we have $\gamma(F,r,\mathcal F_{u}(I))\ge \alpha/2-\varepsilon$ for every small $\varepsilon >0$. Hence $\gamma(F,r,\mathcal F_{u}(I))\ge\alpha/2=r/3$.
 Now, it follows that
$$\underline{p}_{{\mathcal F}_{eq}(I)}(\mathcal{F}_u(I), F) = 2 \liminf_{r \to 0^+} \frac{\gamma(F, r, \mathcal{F}_u(I))}{r} \ge \frac{2}{3}.$$
Finally, $\underline{p}_{{\mathcal F}_{eq}(I)}(\mathcal F_u(I))\ge 2/3$.
\end{proof}

Let $\mathbb Q$ stand for the set of all rationals. 

\begin{lm} \label{L2}
Let $F=(f_n)_{n\in\mathbb N}\in\mathcal F_{eq}(I)$ and fix an enumeration $(r_n)_{n\in\mathbb N}$ of $\mathbb Q\cap I$. Assume that $c>\beta>0$. If $H=(h_n)_{n\in\mathbb N}\in\mathcal F(I)$ is such that for all $n\in\mathbb N$
$$\mid h_n(r_n)-f_n(r_n)\mid\,\ge c ,\mbox{ and } \mid h_n(x)-f_n(x)\mid\,\le\beta\mbox{ if } x\in I\setminus\{r_n\},$$
then $H$ is not equi-$B_1$.
\end{lm}
\begin{proof}
Suppose that $H$ is equi-$B_1$. Choose $\varepsilon\in (0,(c-\beta)/2)$ and pick the respective gauges 
$\delta_F$ and $\delta_H$ for $\varepsilon$ witnessing that $F$ and $H$ are equi-$B_1$ sequences.
Let $\delta:=\min\{ \delta_F,\delta_H\}$. For  each $k\in\mathbb N$ let 
$A_k:=\{ x\in I\colon \delta(x)>1/k\}$. Since $I=\bigcup_{k\in\mathbb N} {\cl}(A_k)$, 
there is $k\in\mathbb N$ such that the interior $U$ of ${\cl}(A_k)$ is nonempty, by the Baire Category Theorem. Choose $r_n\in U$ and $x\in A_k\cap U\setminus\{ r_n\}$ such that 
$|x-r_n|<\min\{\delta(r_n),1/k\}$. Then $|x-r_n|<\min\{\delta(x),\delta(r_n)\}$ and so,
$|f_n(x)-f_n(r_n)|<\varepsilon$. Consequently, by the assumptions,
$$|h_n(r_n)-h_n(x)|\ge |h_n(r_n)-f_n(r_n)|-|f_n(r_n)-f_n(x)|-|f_n(x)-h_n(x)|
\ge c-\varepsilon-\beta>\varepsilon,$$
a contradiction.
\end{proof}

Our next result improves Theorem 3.6 in \cite{Ali}.

\begin{theorem}
The family $\mathcal F_{eq}(I)$ is lower $2/3$-porous in $\mathcal F_{p.w.}(I)$.
\end{theorem}
\begin{proof}
At first let us show that  $\mathcal F_{eq}(I)$ is at least lower $2/3$-porous in $\mathcal F_{p.w.}(I)$. For this purpose consider a sequence $F\in \mathcal F_{eq}(I)$, $F = (f_n)_{n\in\mathbb N}$, and let
$f_n\to f$ pointwise on $I$. Fix numbers $r\in(0,1/2)$, $\varepsilon\in\left(0,r/3\right)$. 
Set $\alpha:=2r/3$ and fix an enumeration 
$(r_n)_{n\in\mathbb N}$ of $\mathbb Q\cap I$.
Then for each $n\in\mathbb N$ define $g_n\colon I\to \mathbb R$ as follows
\[
g_n(x):=\left\{
\begin{array}{lll}
f_n(r_n)+\alpha, & \mbox{if } &  x=r_n,\;   f_n(r_n)\le 1/2,\\
f_n(r_n)-\alpha, & \mbox{if } & x=r_n, \; f_n(r_n) >1/2,\\
f_n(x),          & \mbox{if } & x\neq r_n. \\ 
\end{array} \right.
\] 
Note that $g_n[I]\subset I$ and $g_n$ is Baire-one as a modification of $f_n$ in at most one point. Also, $g_n\to f$ pointwise. Hence 
$G:=(g_n)_{n\in\mathbb N}\in\mathcal F_{p.w.}(I)$. Note that $\|G-F\|_\infty=\alpha$.

Now, fix $H:=(h_n)_{n\in\mathbb N}$ in $B(G,\alpha/2 - \varepsilon)\cap \mathcal F_{p.w.}(I)$. We will show that $B(G, \alpha/2 - \varepsilon) \cap \mathcal{F}_{eq}(I) = \emptyset$. For 
each $n\in\mathbb N$ we have
$$|h_n(r_n)-f_n(r_n)|\ge |f_n(r_n)-g_n(r_n)|-|g_n(r_n)-h_n(r_n)|\ge \alpha-\alpha/2+\varepsilon=\alpha/2+\varepsilon$$
and
$$|h_n(x)-f_n(x)|=|h_n(x)-g_n(x)|<\alpha/2-\varepsilon\mbox{ if }x\in I\setminus\{r_n\}.$$
Now, we can use Lemma \ref{L2} with $c:=\alpha/2+\varepsilon$ and $\beta:=\alpha/2-\varepsilon$.
Hence $H\notin\mathcal F_{eq}(I)$.
Since $\varepsilon$ can be arbitrarily small, we get $B(G,\alpha/2)\cap \mathcal F_{eq}(I)=\emptyset$.

Now, we can repeat the analogous reasoning from the proof of Theorem \ref{T1}.
We observe that $B(G,r/3)=B(G,\alpha/2)\subset B(F,r)$ and then it follows that 
$\gamma(F,r,\mathcal F_{eq}(I))\ge r/3$. Hence
$$\underline{p}_{{\mathcal F}_{p.w.}(I)}\left(\mathcal F_{eq}(I),F\right)
=2\liminf\limits_{r\rightarrow 0^{+}}\frac{\gamma\left(F,r,\mathcal F_{eq}(I)\right)}{r}\ge \frac{2}{3}.
$$
Thus $\mathcal F_{eq}(I)$ is at least lower $2/3$-porous in $\mathcal F_{p.w.}(I)$. 

To complete the proof, we use Lemma \ref{L_e} with $M:=\mathcal F_{eq}(I)$ and $Z:=\mathcal F_{p.w.}(I)$, and we show that $\mathcal F_{eq}(I)$ is lower $2/3$-porous in $\mathcal F_{p.w.}(I)$.
\end{proof}

\begin{theorem} \label{TTw}
The family $\mathcal F_{p.w.}(I)$ is lower $2/3$-porous in $\mathcal F(I)$.
\end{theorem}
\begin{proof}
First let us show that $\mathcal{F}_{p.w.}(I)$ is at least lower $2/3$-porous in $\mathcal{F}(I)$. To this end, fix a sequence $F = (f_n)_{n\in\mathbb{N}} \in \mathcal{F}_{p.w.}(I)$ and any $r \in (0, 1/2)$.  
Let $f$ denote the pointwise limit of $(f_n)_{n\in\mathbb N}$. Since $f$ is Baire-one,  its set $C(f)$ of continuity points is dense in $I$.  

Let $\varepsilon \in (0, r/3)$ and set $\alpha: = 2r/3$. Let $(r_n)_{n\in\mathbb N}$ denote an enumeration of $\mathbb Q\cap I$. Define 
$s\colon \mathbb N\to\{-1,1\}$ as follows
$$s(i):=\begin{cases}
1, & \text{if } f(r_i)\le \frac{1}{2}, \\ 
-1, & \text{if } f(r_i) >\frac{1}{2}. \\ 
\end{cases}
$$
Since $f_n\to f$ pointwise on $I$, for each $i\in\mathbb N$ pick an index $N_i$ such that
$|f_n(r_i)-f(r_i)|<\varepsilon$ for each $n\ge N_i$. For any $n\in\mathbb N$ define
$$g_n:=
f_n+\alpha\sum\limits_{i\le n,\, N_i\le n}s(i)\chi_{\{r_i\}}.
$$
Let $G:=(g_n)_{n\in\mathbb N}$. We observe that $G\in\mathcal F(I)$.
Indeed, $g_n$ is Baire-one for each $n\in\mathbb N$, and for a fixed $k\in\mathbb N$ we have
$$g_n(r_k)=f_n(r_k)+\alpha s(k)\mbox{ for each }n\ge\max\{k,N_k\}.$$
Hence $g_n\to g$ pointwise
where $g:=f+\alpha\sum_{i=1}^\infty s(i)\chi_{\{r_i\}}$. Note that the values of $g_n$ are in $I$
since, if a point $r_i$ with $s(i)=1$ is fixed, for any $n\ge N_i$ we have
$$g_n(r_i)=f_n(r_i)+\alpha<f(r_i)+\varepsilon +1/3\le 1/2+1/6+1/3=1,$$
and similarly, if $s(i)=-1$. Also, note that $\|G-F\|_\infty=\alpha$.
This and $3\alpha/2=r$ imply that $B(G,\alpha/2-\varepsilon)\subset B(F,r)$.

We will show that $B(G,\alpha/2-\varepsilon)\cap\mathcal F_{p.w.}(I)=\emptyset$.
So, take $H \in B(G, \alpha/2 - \varepsilon) \cap \mathcal{F}(I)$ where $H = (h_n)_{n \in \mathbb{N}}$ is pointwise convergent to $h\colon I\to I$. 
Fix $x_0\in C(f)$ and pick a neighbourhood $V_0$ in $I$ of $x_0$ such that the oscillation of $f$ on $V_0$ is less than $\varepsilon$. 
Take any neighbourhood $V\subset V_0$ of $x_0$ and any $x\in V\cap\mathbb Q$, $x'\in V\setminus\mathbb Q$. Let $x=r_k$. Then for any 
$n\ge \max\{ k,N_k\}$ we have
$$g_n(x)-g_n(x') = f_n(x)+\alpha s(k)-f_n(x')$$
and
$$\mid g_n(x)-g_n(x')\mid\,\le\,\mid g_n(x)-h_n(x)\mid\, +\mid h_n(x)-h_n(x')\mid\, +\mid h_n(x')-g_n(x')\mid$$
$$ \le  2\|H-G\|_\infty\, + \mid h_n(x)-h_n(x')\mid\, \le \alpha-2\varepsilon +\,\mid h_n(x)-h_n(x')\mid . $$
Taking $n\to\infty$, we obtain
$$\mid g(x)-g(x')\mid\, \ge \alpha -\varepsilon$$
and
$$\mid g(x)-g(x')\mid\,\le\,\alpha-2\varepsilon+\mid h(x)-h(x')\mid .$$
So,
$\mid h(x)-h(x')\mid\,\ge \varepsilon$. Hence $h$ is discontinuous at $x_0$. It follows that $C(h)\cap C(f)=\emptyset$. Suppose that $h$ is Baire-one. Then $C(f)$ and $C(h)$ are disjoint dense $G_\delta$ sets in $I$ which contradicts the Baire Category Theorem. Thus $h$ is not Baire-one.
Consequently, $B(G, \alpha/2-\varepsilon)\cap\mathcal F_{p.w.}(I)=\emptyset$. 

Observe that our choice of $G$ depends on $\varepsilon$, so we write $G= G_\varepsilon$. Then  
we can repeat the respective argument from the proof of Theorem \ref{T1} to obtain the following
$$\underline{p}_{\mathcal F(I)}\left(\mathcal F_{p.w.}(I),F\right)=2\liminf\limits_{r\rightarrow 0^{+}}\frac{\gamma\left(F,r,\mathcal F_{p.w.}(I)\right)}{r}\ge\frac{2}{3}.$$
Hence $\mathcal F_{p.w.}(I)$ is at least lower $2/3$-porous in $\mathcal F(I)$. Moreover, by Lemma \ref{L_e} we deduce that $\mathcal F_{p.w.}(I)$ is lower $2/3$-porous in $\mathcal F(I)$.
\end{proof}

We can extend the chain of inclusions (\ref{inclu}) by adding the largest family $\mathcal F$. One can ask how small is $\mathcal F(I)$ in $\mathcal F$. Answering this question, we state the following result.

\begin{theorem}
The family $\mathcal F(I)$ is lower $2/3$-porous in $\mathcal F$.
\end{theorem}
\begin{proof}
First observe that we can easily modify a given sequence $(f_n)_{n\in\mathbb N}\in\mathcal F(I)$ to obtain $(g_n)_{n\in\mathbb N}\in\mathcal F\setminus\mathcal F(I)$. It suffices to change the values of $f_n$'s at one point $x_0$ to ensure that $\lim_n f_n(x_0)$ does not exist. Recall that, if a Baire-one function on $I$ is changed at one point, then it remains Baire-one.

Now, fix 
$F=(f_n)_{n\in\mathbb N}\in\mathcal F(I)$ and $r\in(0,1/2)$. 
Set $\alpha:=2r/3$. For each $n\in\mathbb N$, we define $g_n\colon I\to I$ as follows. Let
$g_n(x):=f_n(x)$ for all $n\in\mathbb N$ and $x\neq 0$. Next, for any even $n\in\mathbb N$ let
$$g_n(0):=\begin{cases} f_n(0) + \alpha, & \text{ if } f_n(0)\in[0,1/2],\\ 
f_n(0) - \alpha, & \text{ if } f_n(0)\in (1/2,1], 
\end{cases}
$$
and let $g_n(0):=f_n(0)$ if $n\in\mathbb N$ is odd.
Then $g_n$'s are Baire-one functions from $I$ to $I$. So, $G:=(g_n)_{n\in\mathbb N}\in\mathcal F$.
Note that $\lim_n g_n(0)$ does not exist since $(g_{2k-1}(0))_{k\in\mathbb N}$ has the limit $c:=\lim_n f_n(0)$ and each limit point of
$(g_{2k}(0))_{k\in\mathbb N}$ is equal to $c\pm\alpha\neq c$.
We have $\|G-F\|_\infty=\alpha$ since $\mid g_k(0)-f_k(0)\mid\,=\alpha$ for some even index $k$.
From $\|G-F\|_\infty=\alpha$ we infer that $B(G,r/3)=B(G,\alpha/2)\subset B(F,r)$.

Let $\varepsilon\in(0,r/3)$. We will show that $B(G,\alpha/2-\varepsilon)\cap\mathcal F(I)=\emptyset$. So, take 
$H\in B(G,\alpha/2-\varepsilon)$ with $H=(h_n)_{n\in\mathbb N}$. 
Since $F\in\mathcal F(I)$, $(f_n)_{n\in\mathbb N}$ is pointwise convergent.
In particular, $(f_n(0))_{n\in\mathbb N}$ satisfies the Cauchy condition. 
So, pick an index $k_0$ such that $\mid f_{2k}(0)-f_{2k-1}(0)\mid\,\le \varepsilon$ for each $k\ge k_0$.

Thus for every $k\ge k_0$ we have
$$\mid g_{2k}(0)-g_{2k-1}(0)\mid\, \ge\alpha\,-\mid f_{2k}(0)-f_{2k-1}(0)\mid\,\ge \alpha-\varepsilon$$
and
$$\mid g_{2k}(0)-g_{2k-1}(0)\mid\,\le\,\mid g_{2k}(0)-h_{2k}(0)\mid
+\mid h_{2k}(0)-h_{2k-1}(0)\mid + \mid h_{2k-1}(0)-g_{2k-1}(0)\mid$$
$$\le 2\|H-G\|_\infty + \mid h_{2k}(0)-h_{2k-1}(0)\mid\,\le\alpha-2\varepsilon + \mid h_{2k}(0)-h_{2k-1}(0)\mid. $$
Hence $\mid h_{2k}(0)-h_{2k-1}(0)\mid\,\ge \varepsilon$ for each $k\ge k_0$. Thus $\lim_n h_n(0)$ does not exist.

Consequently, $B(G,\alpha/2-\varepsilon)\cap\mathcal F(I)=\emptyset$.
This implies that $\mathcal F(I)$ is at least lower $2/3$-porous in $\mathcal F$. Finally, Lemma \ref{L_e}
yields that it is lower $2/3$-porous in $\mathcal F$.
\end{proof}


\section{Porosity results for subspaces of $b\mathcal F$}

In this section, we consider the family of all Baire class-one functions $f\colon I \to \mathbb{R}$, denoted by $B_1(I, \mathbb{R})$, as well as the corresponding family $b\mathcal{F}$ of all uniformly bounded sequences of functions in $B_1(I, \mathbb{R})$.  Consequently, we do not assume that the values of Baire-one functions in a sequence belong to $I$. In contrast to $\mathcal F$, the family $b\mathcal F$ has a linear structure. Recall that $b\mathcal{F}$ is a normed space with the norm 
$\|\cdot\|_\infty$ defined at the end of Section 1.
In fact $b\mathcal F$ is a Banach space since its completeness can be proved in the same way as in the case of Proposition \ref{PP1}.

We modify the objects considered in \cite{Ali1} and \cite{Ali}. Similarly as before, we define the following subfamilies of $b\mathcal{F}$:
\begin{itemize}
    \item \textbf{$b\mathcal{F}_u(I, \mathbb{R})$} -- the family of all uniformly convergent sequences with Baire-one terms;
    \item \textbf{$b\mathcal{F}_{eq}(I, \mathbb{R})$} -- the family of pointwise convergent equi-$B_1$ sequences;
    \item \textbf{$b\mathcal{F}_{p.w.}(I, \mathbb{R})$} -- the family of all sequences with Baire-one terms that converge pointwise and their limits are Baire-one functions;
    \item \textbf{$b\mathcal{F}(I, \mathbb{R})$} -- the family of pointwise convergent sequences with Baire-one terms.
\end{itemize}

Observe that
\begin{equation} \label{eqi}
b\mathcal{F}_u(I, \mathbb{R})\subset b\mathcal{F}_{eq}(I, \mathbb{R})\subset b\mathcal{F}_{p.w.}(I, \mathbb{R})\subset b\mathcal{F}(I, \mathbb{R})\subset b\mathcal F,
\end{equation}
and all these inclusions are proper.

\begin{prop} \label{PP2}
All the subspaces of $b\mathcal F$ listed in (\ref{eqi}) are closed linear spaces. Consequently, they are Banach spaces.
\end{prop} 
\begin{proof}
The fact that the considered subspaces are linear follows directly from their definitions. For the space
$b\mathcal F_{eq}$ we use Remark \ref{rema}. Of course, the uniform boundedness of sequences is preserved under operations $F+G$ and $\alpha F$. 

To show that $b\mathcal F(I,\mathbb R)$ is closed, we consider a sequence $(F_\ell)_{\ell\in\mathbb N}$ in $b\mathcal F(I,\mathbb R)$ convergent to $F=(f_n)_{n\in\mathbb N}\in b\mathcal F$. As in the second part of the proof of Proposition \ref{PP1}, we show that $F$ is a pointwise convergent sequence. Note that it is uniformly bounded by the standard argument using the norm of $b\mathcal F$.

The closedness of the remaining subspaces of $b\mathcal F$, listed in (\ref{eqi}), can be proved in a similar way as it was done for the respective subsets of $\mathcal F$ in \cite{Ali}. 

Since all these subspaces of $b\mathcal F$ are closed and linear, they all are Banach spaces.
\end{proof}

Clearly, if $Y$ is a closed proper linear subspace of a normed space $X$, then $Y$ is nowhere dense in $X$.
Indeed, otherwise $Y$ would contain a ball with centre zero which implies that $Y=X$.
So, every space different from $b\mathcal F$ in the chain (\ref{eqi}) is a nowhere dense subset of the next space of this chain. What can we say about their porosity?

Recall some notions of porosity in normed spaces.
Let $(X, \|\cdot\|)$ be a real normed linear space and $M \subset X$. 
By $X^*$ we denote the space of all continuous linear functionals on $X$.

\begin{df} [\cite{PreZaj}, \cite{Zaj1}] 
Let $R > 0$. We say that $M\subset X$ is $R$-ball porous if for any $x \in M$ and $\alpha \in (0,1)$, there exists $y \in X$ such that $\|x - y\| = R$ and $B(y,\alpha R) \cap M = \emptyset$.
\end{df}

\begin{remark} \label{remi}
Observe that, if $M\subset X$ is $R$-ball porous for every $R>0$, then $M$ is lower 1-porous. Indeed, fix
$x\in M$ and $r>0$. Using $r/2$-ball porosity of $M$, for each $\alpha\in (0,1)$ we can find $y\in X$ with $\|x-y\|=r/2$ such that 
$B(y,\alpha r/2)\cap M=\emptyset$. Since $B(y,\alpha r/2)\subset B(x,r)$,
we can take $\alpha\to 1$ which implies that $\underline{p}(M,x)=1$ by the arbitrariness of $r$.
Hence $\underline{p}(M)=1$ by the arbitrariness of $x$.
\end{remark}

\begin{df} [\cite{Strobin}]
We say that $M$ is $c$-porous if for any $x \in X$ and every $r > 0$, there are $y \in B(x,r)$ and $\phi \in X^* \setminus \{0\}$ such that
$$\{z \in X : \phi(z) > \phi(y)\} \cap M = \emptyset. $$
\end{df}

\begin{theorem} [\cite{Strobin}, Propositions 2.5 and 2.8] \label{FS}
A subset $M$ of a normed space $X$ is $c$-porous if and only if the convex hull ${\rm conv} M$ is nowhere dense.
If a set $M$ is $c$-porous then it is R-ball porous for every $R>0$.
\end{theorem}

\begin{cor}
Every space different from $b\mathcal F$ in the chain (\ref{eqi}) is c-porous in the next space of this chain. Consequently, it is $R$-ball porous for every $R>0$ and it is lower 1-porous.
\end{cor}
\begin{proof}
Fix a space $Y$ different from $b\mathcal F$ in the chain (\ref{eqi}). By Proposition \ref{PP2}
it is closed in the next space $Z$ of this chain. Thus we conclude that $Y$ is nowhere dense in $Z$,
and of course it is convex. Thus by Theorem \ref{FS} it is c-porous. Then it suffices to use the second
assertion of Theorem \ref{FS} and Remark \ref{remi}.
\end{proof}

Hence we see that, in Section 3, we have obtained a stronger kind of porosity for the respective sequence spaces, in comparison with the lower $2/3$-porosity of the respective spaces in Section 2.
This could be surprising at the first glance. The reason is that, in Section 2, we deal with functions taking values in $[0,1]$ which causes a kind of additional boundary conditions. In fact Lemma \ref{L_e} cannot be repeated for real-valued Baire-one functions on~$I$.

\hspace{0.3 em}

\noindent
{\bf Acknowledgements.} We would like to thank Piotr Szuca who has found some gaps in the previous versions of the paper and helped us to make several improvements.

\end{document}